\documentclass[draft]{amsart}
\usepackage{amsmath, amssymb}

\newtheorem{thm}{Theorem}
\newtheorem{lem}[thm]{Lemma}
\newtheorem{cor}[thm]{Corollary}
\newtheorem*{conj}{Conjecture}

\theoremstyle{remark}
\newtheorem{rmk}[thm]{Remark}

\numberwithin{thm}{section}
\numberwithin{equation}{section}

\newcommand{\cv}{\mathbf{C}}
\newcommand{\zv}{\mathbf{Z}}

\newcommand{\rv}{\mathbf{R}}
\newcommand{\nv}{\mathbf{N}}
\newcommand{\aut}{\textup{Aut}}

\begin{document}

\subjclass[2010]{32A07, 32H02}
\keywords{quasi-circular domain; automorphism}

\title{On automorphisms of quasi-circular domains fixing the origin}

\author{Feng Rong}

\address{Department of Mathematics, Shanghai Jiao Tong University, 800 Dong Chuan Road, Shanghai, 200240, P.R. China}
\email{frong@sjtu.edu.cn}

\thanks{The author is partially supported by the National Natural Science Foundation of China (Grant No. 11371246).}

\begin{abstract}
It is known that automorphisms of quasi-circular domains fixing the origin are polynomial mappings. By introducing the so-called resonance order and quasi-resonance order, we provide a uniform upper bound for the degree of such polynomial automorphisms. As a particular consequence of our result, we obtain a generalization of the classical Cartan's theorem for circular domains to the quasi-circular case.
\end{abstract}

\maketitle

\section{Introduction}

Let $D$ be a bounded domain in $\cv^n$ containing the origin. We say that $D$ is a \textit{quasi-circular} domain if $D$ is invariant under the mapping
$$\rho:\cv^n\rightarrow \cv^n;\ \ \ (z_1,\cdots,z_n)\mapsto (e^{im_1\theta}z_1,\cdots,e^{im_n\theta}z_n),\ \ \ \theta\in \rv,$$
where $m_i\in \zv^+$. If $m_1=\cdots=m_n$, $D$ is called a circular domain. We call $(m_1,\cdots,m_n)$ the \textit{weight} of $D$. We will assume without loss of generality that $m_1\le m_2\le \cdots \le m_n$ and $\textup{gcd}(m_1,\cdots,m_n)=1$.

Let $f$ be an automorphism of $D$ fixing the origin. When $D$ is a circular domain, the well-known Cartan's theorem asserts that $f$ must be a linear mapping (\cite {Cartan}, see also \cite[Ch. 5, Prop. 2]{N:SCV}). In \cite{K:auto}, Kaup showed that in the quasi-circular case $f$ must be a polynomial mapping. Our main result gives a uniform upper bound for the degree of such polynomial mappings.

\begin{thm}\label{T:main}
Let $D$ be a bounded quasi-circular domain in $\cv^n$ containing the origin and $f$ an automorphism of $D$ fixing the origin. Then $f$ is a polynomial mapping with degree less than or equal to the quasi-resonance order. More precisely, writing $f=(f_1,\cdots,f_n)$, the degree of each $f_i$ is less than or equal to the $i$-th quasi-resonance order.
\end{thm}

As a corollary to our main result above, we have the following linearity result, which is a generalization of the classical Cartan's theorem for circular domains to the quasi-circular domains.

\begin{cor}\label{C:linear}
Let $D$ be a bounded quasi-circular domain in $\cv^n$ containing the origin and $f$ an automorphism of $D$ fixing the origin. If the weight of $D$ has resonance order equal to one, then $f$ is linear.
\end{cor}

We will give the relevant definitions and the proofs of Theorem \ref{T:main} and Corollary \ref{C:linear} in the next section. For generalizations of Cartan's theorem to proper holomorphic maps, see e.g. \cite{B:proper, BP:proper, Bo:quasi, Bo:quasi1}. For other results concerning quasi-circular domains, see e.g. \cite{CPS:quasi, Ko:quasi, Ko:quasi1} and the references therein.

We thank Fusheng Deng for pointing out some mistakes in an earlier draft of this paper and Atsushi Yamamori for his comments, especially on Corollary \ref{C:BKU}. We also thank the referee for his/her valuable comments which improved the presentation of the paper.

\section{Automorphisms fixing the origin}

Let $f\in \aut(D)$ with $f(0)=0$. Write $f=(f_1,\cdots,f_n)$ with
$$f_i(z)=\sum_{|\alpha|\ge 1} a_\alpha^i z^\alpha,\ \ \ 1\le i\le n,$$
where $\alpha=(\alpha_1,\cdots,\alpha_n)$, $|\alpha|=\alpha_1+\cdots+\alpha_n$ and $z^\alpha=z_1^{\alpha_1}\cdots z_n^{\alpha_n}$.

Write $F=f^{-1}$ and let $u=\det [\partial f_i/\partial z_j]$ and $U=\det [\partial F_i/\partial z_j]$ be the holomorphic Jacobian determinant of $f$ and $F$ respectively. Note that $|u|^2$ (resp. $|U|^2$) is the real Jacobian of $f$ (resp. $F$), and $U(f(z))=u(z)^{-1}$. Denote by $\langle\ ,\ \rangle_D$ the $L^2$ inner product on bounded domains $D$. By a standard argument using the change of variables formula, we have
\begin{equation}\label{E:inner}
\langle u(\phi\circ f),\psi\rangle_D=\langle\phi,U(\psi\circ F)\rangle_D,
\end{equation}
where $\phi, \psi\in L^2(D)$.

A polynomial $p$ is said to be $m$-homogeneous of order $k$ if $p(\lambda^{m_1}z_1,\cdots,\lambda^{m_n}z_n)=\lambda^k p(z_1,\cdots,z_n)$ for any $\lambda\in \cv$. Denote $m\cdot \alpha=m_1\alpha_1+\cdots+m_n\alpha_n$. We have the following lemma. The proof is similar to that of \cite[Lemma 2]{B:ML}, which we include for completeness and clarity.

\begin{lem}\label{L:basis}
Let $D$ be a bounded quasi-circular domain containing the origin with the weight $m$. Then for each multi-index $\alpha$, there exists a $m$-homogeneous polynomial $p_\alpha$ of order $m\cdot \alpha$, such that
\begin{equation}\label{E:basis}
\langle h,p_\alpha\rangle_D=\frac{\partial^\alpha}{\partial z^\alpha} h(0),
\end{equation}
for all holomorphic functions $h$ in $L^2(D)$.
\end{lem}
\begin{proof}
Denote by $H^2(D)$ the subspace of $L^2(D)$ consisting of holomorphic functions. Since evaluating the derivative $\frac{\partial^\alpha}{\partial z^\alpha} h(0)$ is a continuous functional on $H^2(D)$, the Riesz representation theorem yields a unique holomorphic function $p_\alpha$ satisfying \eqref{E:basis}. By \eqref{E:inner} (taking $F=\rho$), we have
$$\langle h,p_\alpha\circ \rho\rangle_D=\langle h\circ \rho^{-1},p_\alpha\rangle_D=\frac{\partial^\alpha}{\partial z^\alpha} h\circ \rho^{-1}(0)=e^{-im\cdot \alpha \theta} \frac{\partial^\alpha}{\partial z^\alpha} h(0),$$
which together with the uniqueness of $p_\alpha$ implies
$$p_\alpha(e^{im_1\theta}z_1,\cdots,e^{im_n\theta}z_n)=e^{im\cdot \alpha\theta}p_\alpha(z_1,\cdots,z_n),$$
for all $z\in D$ and $\theta\in \rv$.

A standard argument, using the function $q(\tau)=\tau^{-m\cdot \alpha}p_\alpha(\tau^{m_1}z_1,\cdots,\tau^{m_n}z_n)$, shows that $p_\alpha(z)$ is a $m$-homogeneous polynomial of order $m\cdot \alpha$. This completes the proof.
\end{proof}

Let $S$ be the linear span of the set of polynomials $p_\alpha$ as in Lemma \ref{L:basis}. If $h\in H^2(D)$ is orthogonal to all $p_\alpha$, then $h$ vanishes to infinite order at the origin, and thus must be identically zero. Moreover, by \eqref{E:inner} and $m$-homogenity of $p_\alpha$, we have $\langle p_\alpha,p_\beta\rangle=0$ if $m\cdot \alpha\neq m\cdot \beta$. Hence, $S$ is equal to the set of all polynomials and each monomial $z^\alpha$ can be written as
\begin{equation}\label{E:alpha1}
z^\alpha=\sum_{m\cdot \beta=m\cdot \alpha} c_\beta^\alpha p_\beta.
\end{equation}

Before we proceed, let us first make some definitions.

For $1\le i\le n$, define the \textit{i-th resonance set} as
$$E_i:=\{\alpha:\ m\cdot \alpha=m_i\},$$
and the \textit{i-th resonance order} as
$$\mu_i:=\max\{|\alpha|:\ \alpha\in E_i\}.$$
Note that $\mu_i\le m_i$. Then, define the \textit{resonance set} as
$$E:=\bigcup\limits_{i=1}^n E_i,$$
and the \textit{resonance order} as
$$\mu:=\max\{|\alpha|:\ \alpha\in E\}=\max\limits_{1\le i\le n} \mu_i.$$

Denote by $\Gamma$ the set of all multi-indices $\alpha$. For every $\alpha\in \Gamma$, we can write
$$p_\alpha=\sum_{m\cdot \beta=m\cdot \alpha} b_\beta^\alpha z^\beta.$$
Set $d(\alpha)$ to be the minimum of $|\beta|$ in the above expression for $p_\alpha$. For each $k\in \nv$, define
$$P_k:=\{\alpha:\ d(\alpha)>k\},\ \ \ Q_k:=\Gamma\backslash P_k.$$
We then define the \textit{i-th quasi-resonance order} as
$$\nu_i:=\max\{|\alpha|:\ \alpha\in Q_{\mu_i}\},$$
and the \textit{quasi-resonance order} as
$$\nu:=\max\limits_{1\le i\le n} \nu_i.$$

Now for any $f\in \aut(D)$ with $f(0)=0$ and $F=f^{-1}$, by \eqref{E:inner}, we have
$$\langle u,z^\alpha\rangle_D=\langle1,UF^\alpha\rangle_D=U(0)F(0)^\alpha.$$
Since $F(0)=0$, we have $\langle u,z^\alpha\rangle_D=0$ for any $|\alpha|>0$. Hence, we get $u$ is a constant. For each $1\le i\le n$, write
$$z_i=\sum_{\beta\in E_i} c_\beta^i p_\beta.$$
Then, by \eqref{E:inner} and \eqref{E:alpha1}, we have
$$\langle uf_i,z^\alpha\rangle_D=\langle z_i,UF^\alpha\rangle_D=\langle\sum_{\beta\in E_i} c_\beta^i p_\beta,UF^\alpha\rangle_D=\sum_{\beta\in E_i} c_\beta^i \frac{\partial^\beta}{\partial z^\beta} \{UF^\alpha\}(0).$$
Since $F(0)=0$ and $u$ is a constant, we get
$$\langle f_i,z^\alpha\rangle_D=0,\ \ \ |\alpha|>\mu_i.$$ 
By the definition of $P_k$, we get
$$\langle f_i,p_\alpha\rangle_D=0,\ \ \ \alpha\in P_{\mu_i}.$$ 
Therefore, the degree of $f_i$ is less than or equal to the $i$-th quasi-resonance order. This completes the proof of Theorem \ref{T:main}.

\begin{rmk}
The above provides another proof of Kaup's result that automorphisms of quasi-circular domains fixing the origin are polynomial mappings.
\end{rmk}

From the definitions, it is easy to see that $E_i\subset Q_{\mu_i}$ and hence $\nu_i\ge \mu_i$ for each $i$. Nonetheless, we have

\begin{lem}
$Q_1\subset E$.
\end{lem}
\begin{proof}
Assume $\alpha\in Q_1$ and write $p_\alpha=\sum\limits_{m\cdot \beta=m\cdot \alpha} b_\beta^\alpha z^\beta$. Since $\alpha\in Q_1$, at least one $z^\beta$ in the summands of $p_\alpha$ has $|\beta|=1$. Suppose $\beta=(0,\cdots,0,1,0,\cdots,0)$ with the $i$-th entry being 1. Then $m\cdot\beta=m_i$ and we get $\alpha\in E_i$. The lemma follows.
\end{proof}

As an immediate corollary, we have

\begin{cor}\label{C:1}
If $\mu=1$, then $\nu=1$.
\end{cor}

Combining Theorem \ref{T:main} and Corollary \ref{C:1}, we get Corollary \ref{C:linear}.

\begin{rmk}\label{R:Cartan}
When $D$ is circular, one easily sees that the resonance order is one. Hence Corollary \ref{C:linear} implies Cartan's theorem in the circular case.
\end{rmk}

\begin{rmk}
In \cite{Y:quasi}, Yamamori showed that for \textit{normal} quasi-circular domains in $\cv^2$, $f$ must be linear. Here $D$ is said to be normal if $2\le m_1\le m_2$ and $\textup{gcd}(m_1, m_2)=1$. In \cite{Y:quasi1}, Yamamori gave further conditions for $f$ to be linear in higher dimensions. When $D$ is a normal quasi-circular domain in $\cv^2$, it is easy to see that the resonance order is one. Hence Corollary \ref{C:linear} implies Yamamori's result in \cite{Y:quasi}. Moreover, as noted in \cite{Y:quasi1}, the linearity conditions given there for higher dimensions are sufficient conditions for the resonance order to be one.
\end{rmk}

\begin{rmk}\label{R:commute}
If an automorphism $f=(f_1,\cdots,f_n)$ of a quasi-circular domain commutes with the quasi-circular action $\rho$, then the degree of each $f_i$ is actually less than or equal to the $i$-th resonance order. For a discussion on mappings commuting with more general torus actions, see \cite{Ra:Torus}.
\end{rmk}

\begin{rmk}\label{R:two}
With the same proof, Theorem \ref{T:main} and Corollary \ref{C:linear} extend to biholomorphic mappings between two bounded quasi-circular domains $D_1$ and $D_2$ with the same weight, which maps the origin to the origin.
\end{rmk}

As a corollary, we have the following Braun-Kaup-Upmeier type result (cf. \cite{BKU:auto}).

\begin{cor}\label{C:BKU}
Let $D_1$ and $D_2$ be two bounded quasi-circular domains containing the origin. Assume that they have the same weight with the resonance order being one. Then $D_1$ and $D_2$ are biholomorphically equivalent if and only if they are linearly equivalent.
\end{cor}
\begin{proof}
Clearly, we only need to show that the existence of a biholomorphism between $D_1$ and $D_2$ implies the existence of a \textit{linear} biholomorphism. By \cite[Lemma 2]{K:auto}, the existence of a biholomorphism implies the existence of a biholomorphism \textit{fixing the origin}. Then from Remark \ref{R:two}, we know that the latter biholomorphism is linear since the resonance order is one.
\end{proof}

Finally, we make the following conjecture (cf. Remark \ref{R:commute}).

\begin{conj}
Let $D$ be a bounded quasi-circular domain in $\cv^n$ containing the origin and $f$ an automorphism of $D$ fixing the origin. Then $f$ is a polynomial mapping with degree less than or equal to the resonance order. More precisely, writing $f=(f_1,\cdots,f_n)$, the degree of each $f_i$ is less than or equal to the $i$-th resonance order.
\end{conj}

The following example shows that the bounds given in the conjecture would be optimal. For $n\ge 2$ and $q>0$, denote by $E_{q,n}$ the symmetrized ($q,n$)-ellipsoid. Note that $E_{q,n}$ is a quasi-circular domain with weight $(1,\cdots,n)$. In \cite{Z:proper}, Zapa\l{}owski gave a complete description of $\aut(E_{q,n})$. In particular, using \cite[Corollary 4(b)]{Z:proper}, one readily checks that $\aut(E_{1,n})$ contains automorphisms $\phi=(\phi_1,\cdots,\phi_n)$ fixing the origin, where
$$\phi_i(z)=\left(\frac{2}{n}\right)^iC_n^i z_1^i+\sum_{j=1}^i (-1)^j\left(\frac{2}{n}\right)^{i-j}C_n^{i-j}z_1^{i-j}z_j,\ \ \ 1\le i\le n.$$

\end{document}